\documentclass[11pt,regno]{amsart}

\usepackage{overpic}

\usepackage{euscript,graphicx,color}
\usepackage{amsmath}
\usepackage{amssymb}
\usepackage{amsfonts}
\usepackage{mathrsfs}

\newtheorem{theo}{Theorem}[section]
\newtheorem{theo-app}{Theorem}[section]

\newtheorem{theorem}[theo]{Theorem} %dopisalem, numeruje twierdzenia wg rozdzialow.
\newtheorem{corollary}[theo]{Corollary}
\newtheorem{proposition}[theo]{Proposition}

\newtheorem{lemma}[theo]{Lemma}
\theoremstyle{definition}

\newtheorem{remark}[theo]{Remark}
\newtheorem*{remark*}{Remark}
\newtheorem{definition}[theo]{Definition}

\numberwithin{equation}{section}
\renewcommand{\hat}{\widehat}

\newcommand{\sD}{\mathscr{D}}

\newcommand{\sT}{\mathscr{T}}
\newcommand{\sA}{\mathscr{A}}

\DeclareMathOperator{\cl}{cl}

\DeclareMathOperator{\BD}{BD}%Bounded Distortion

\newcommand{\C}{\mathbb{C}}
\newcommand{\Z}{\mathbb{Z}}

\newcommand{\D}{\mathbb{D}}
\newcommand{\mbA}{\mathbb{A}}

\DeclareMathOperator{\var}{var}

\DeclareMathOperator{\Exp}{Exp}
\DeclareMathOperator{\Lyap}{Lyap}
\DeclareMathOperator{\CE2}{CE2}
\DeclareMathOperator{\Per}{Per}

\DeclareMathOperator{\diam}{diam}

\DeclareMathOperator{\dist}{Dist}

\DeclareMathOperator{\Crit}{{\rm Crit}}

\DeclareMathOperator{\Comp}{{\rm Comp}}

\DeclareMathOperator{\tree}{{\rm tree}}
\DeclareMathOperator{\Const}{{\rm Const}}

\DeclareMathOperator{\Fr}{{\rm Fr}}
\DeclareMathOperator{\HD}{{\rm HD}}
\DeclareMathOperator{\hyp}{{\rm hyp}}

\def\bR{\mathbb{R}}

\newcommand{\cA}{\mathscr{A}}

\def\cL{\EuScript{B}}
\def\cC{\EuScript{C}}

\def\cM{\EuScript{M}}

\def\cL{\EuScript{L}}

\def\cB{\mathscr{B}}

\def\sN{\mathcal{N}}

\newcommand{\cH}{\mathcal{H}}

\newcommand{\cP}{\mathcal{P}}

\def\a{\alpha}        \def\La{\Lambda}
\def\z{\zeta}             
\def\e{\varepsilon}      \def\b{\beta}
            \def\d{\delta}
\def\la{\lambda}   

\def\ov{\overline}
 \def\dist{{\rm {dist}}}
 \def\g{\gamma}         \def\s{\sigma}  \def\S{\Sigma}
 \def\Om{\Omega}  \def\La{\Lambda}

\author[F. Przytycki]{Feliks Przytycki$^\dag$} \address{Institute of Mathematics, Polish Academy of Sciences, ul. \'{S}niadeckich 8, 00-656 Warszawa, Poland}
\email{feliksp@impan.pl}

\thanks{Partially supported by Polish NCN Grant % "Chaos, fractals and conformal dynamics, III",
2014/13/B/ST1/01033}

\begin{document}

\date{\today}

\title[Thermodynamic formalism]{Thermodynamic formalism methods in one-dimensional real and complex dynamics}



\maketitle

\begin{abstract}
We survey some results on non-uniform hyperbolicity, geometric pressure and equilibrium states in one-dimensional real and complex dynamics. We present some relations with Hausdorff dimension and measures with refined gauge functions of limit sets for geometric coding trees for rational functions on the Riemann sphere.
We discuss fluctuations of iterated sums of the potential $-t\log |f'|$ and of radial growth of derivative of univalent functions on the unit disc and the boundaries of range domains preserved by a holomorphic map $f$ repelling towards the domains.
\end{abstract}

\tableofcontents

\maketitle
%\setcounter{tocdepth}{1}

%\section{Introduction}

\section{Thermodynamic formalism,  introductory notions}\label{Introduction}

  Among founders of this theory are
    \cite{Sinai}, \cite{Bowen1975} and David Ruelle, who wrote in \cite{Ruelle}: ``thermodynamic formalism  has been developed since G. W. Gibbs to describe [...] physical systems consisting of a large number of subunits''. In particular one considers a {\it configuration space} $\Om$ of functions $\Z^n\to \mbA$ on the lattice $\Z^n$ with interacting values in $\mbA$ over its sites, e.g.~``spin'' values in the Ising model of ferromagnetism. One considers probability distributions on $\Om$, invariant under translation, called {\it equilibrium states} for potential functions on $\Om$.

Given a mapping $f:X\to X$ one considers as a configuration space the set of trajectories $n\mapsto (f^n(x))_{n\in \Z_+}$ or $n\mapsto \Phi(f^n(x))_{n\in \Z_+}$ for a test function $\Phi:X\to Y$.

\

The following simple fact \cite[Lemma 1.1]{Bowen1975} and \cite[Introduction]{Ruelle}, \cite[Introduction]{PUbook}, resulting from Jensen's inequality applied to the function logarithm, stands at the
heart of thermodynamic formalism.

\begin{lemma}[Finite Variational Principle]\label{finite}
For given real
numbers $\phi_1,\dots,\phi_d$, the function
%\begin{equation}\label{equi-finite}
$F(p_1,\dots p_d):=\sum_{i=1}^n -p_i\log p_i+
\sum_{i=1}^d p_i \phi_i$ defined
on the simplex $\{(p_1,\dots ,p_d):p_i\ge 0, \sum_{i=1}^d p_i=1\}$
%\end{equation}
attains its maximum value

\noindent $P(\phi_1,\dots, \phi_d)=\log\sum_{i=1}^de^{\phi_i}$  at
and  only at
%\begin{equation}\label{equi-eq}
$\hat p_j=e^{\phi_j}\bigl(\sum_{i=1}^de^{\phi_i}\bigr)^{-1}.$
%\end{equation}
\end{lemma}

We can read $i\mapsto \phi_i, i=1,\dots ,d$ as a {\it potential} function
and $\hat p_i$ %given by \eqref{equi-eq}
as the equilibrium probability
distribution on the finite space $\{1,\dots ,d\}$.
$P(\phi_1,\dots, \phi_d)$ is called the {\it pressure} or
%and by physicists
{\it free energy}, see \cite{Ruelle}.

\

Let $f:X\to X$ be a continuous mapping of a compact metric space $X$ and $\phi:X\to \mathbb{R}$ be a continuous function (the potential). We define the {\it topological pressure} or free energy by
\begin{definition}\label{top_pres}
\begin{equation}\label{var_pres}
P_{\rm var}(f,\phi)=
      \sup_{\mu\in{\cM}(f)}\left( h_\mu(f)+\int_X \phi \,d\mu\right),
      \end{equation}
where ${\cM}(f)$ is the set of all $f$-invariant Borel probability measures on $X$ and $h_\mu(f)$ is measure theoretical entropy. Sometimes we write $\cM(f,X)$.
\end{definition}
Recall that $h_\mu(f)=\sup_{\mbA} \lim_{n\to\infty} \frac1{n+1} \sum_{A\in {\mbA}^n} -\mu(A)\log\mu(A)$, where the supremum is taken over
finite partitions ${\mbA}$
of $X$, where ${\mbA}^n:=\bigvee_{j=0,\dots, n} f^{-j}{\mbA}$.
Notice that this resembles the sum $\sum_{i=1}^n -p_i\log p_i$ in Lemma \ref{finite}.
%the first sum in \eqref{equi-finite}.

\smallskip

Topological pressure can also be defined in other ways, e.g.~by \eqref{Psep}, and then its equality to the one given by \eqref{var_pres} is called the variational principle. This explains the notation
$P_{\rm var}$.
Any $\mu\in{\cM}(f)$ for which the supremum in \eqref{var_pres} is attained is called {\it equilibrium}, {\it equilibrium measure} or {\it equilibrium state}.

\smallskip

A model case is any map $f:U\to \mathbb{R}^n$ of class $C^1$, defined on a neighbourhood $U$ of a compact set $X\subset \mathbb{R}^n$, {\it expanding} (another name: {\it uniformly expanding} or {\it hyperbolic} in dimension 1)
that is there exist $C>0, \lambda>1$
such that for all positive integers $n$  all $x\in X$ and all $v$ tangent to $\mathbb{R}^n$ at $x$,
\begin{equation}\label{expanding}
||Df^n(v)||\ge C\lambda^n ||v||,
\end{equation}
and {\it repelling} that is every forward trajectory sufficiently close to $X$ must be entirely in $X$. Not assuming the differentiability of $f$ one uses the notion of {\it distance expanding} meaning the increase of distances
under the action of $f$ by a factor at least $\lambda>1$ for pairs of distinct points sufficiently close to each other. Repelling happens to be equivalent to the internal condition: $f|_X$ being an open map,
provided $f$ is open on a neighbourhood of $X$, see \cite[Lemma 6.1.2]{PUbook}. Then the classical theorem holds, here in the version from \cite[Section 5.1]{PUbook}:

\begin{theorem} \label{Gibbs}
Let $f:X\to X$ be a distance expanding, topologically transitive continuous open map of a compact metric space $X$ and  $\phi:X\to\mathbb{R}$ be a  H\"older continuous potential. Then, there exists
%on $X$ an $f$-invariant Borel probability
exactly one measure $\mu_\phi\in\cM(f,X)$, called the {\it Gibbs} measure, satisfying
\begin{equation}\label{Gibbs-eq}
C<\frac{\mu_\phi(f_x^{-n}(B(f^n(x),r_0))}{\exp (S_n\phi(x)-nP(\phi))}<C^{-1}
\end{equation}
where $f_x^{-n}$ is the branch of $f^{-n}$ mapping $f^n(x)$ to $x$ (locally making sense, since $f$ is a local homeomorphism) and $S_n\phi(x):=\sum_{j=0}^{n-1}\phi (f^j(x))$.

%There exists exactly one such $f$-invariant measure, called Gibbs measure.
The measure $\mu_\phi$ is the only equilibrium state for $\phi$.
 It is equivalent to the unique $\phi$-conformal measure $m_\phi$, that is a forward quasi-invariant Borel  probability measure $m_\phi$ with Jacobian
$\exp -(\phi-P(\phi))$. Moreover, the limit
$P(\phi)=P(f,\phi):=$

\noindent $\lim_{n\to\infty}\frac1n\log \sum_{x\in f^{-n}(x_0)}\exp S_n\phi(x)$ exists
and is equal to $P_{\var}(f,\phi)$ for every $x\in X$.

\end{theorem}

This $P(\phi)$ is a normalizing quantity corresponding to $P(\phi_1,\dots, \phi_d)$ in Lemma \ref{finite}
and the sum in the definition of $P(\phi)$
corresponds to the so called {\it statistical sum} over the space $\Om_n$ of all admissible configurations over  $\{0,1,\dots,n-1\}$, as in the Ising model. Compare to the {\it tree pressure} defined in Definition \ref{treep}.

%In fact assuming only that $f$ is an open map and locally expands distance by a factor bounded away from 1 on a compact metric space $X$, is sufficient.
 So $\varsigma:\Sigma^d\to\Sigma^d$, the shift to the left on the space
 $\Sigma^d=\{(\a_0,\a_1,\dots ): \a_j\in\{1,\dots,d\}\}$, defined by $\varsigma((\a_n))=(\a_{n+1})$, is an example where Theorem \ref{Gibbs} holds. The sets $f_x^{-n}(B(f^n(x),r_0)$ correspond to {\it cylinders} of fixed
 $\{\a_j\in\{1,\dots,d\}, j=0,\dots,n-1\}$. One can impose an admissibility condition:
 $\a_i\a_{i+1}$ admissible if the pair has the digit 1 attributed in a defining 0,1 $d\times d$ matrix. Then one calls the system a {\it one-sided topological Markov chain}.

 The condition of openness of $f$ can be replaced by a weaker one: the existence of a finite Markov partition, see \cite{PUbook}.

\smallskip

The existence of a conformal measure follows from the existence of a fixed point in the
convex weakly*-compact set of probability measures for the dual operator to the transfer (Perron-Frobenius-Ruelle) operator $\cL$ divided by the norm, where for $u:X\to\mathbb{R}$ continuous one defines

\begin{equation}\label{transfer}
\cL(u)(x):= \sum_{y\in f^{-1}(x)} u(y)\exp\phi(y).
\end{equation}

Indeed,
for every Borel set $Y\subset X$ on which $f$ is injective, denoting by $I_Y$ indicator function: 1 on $Y$, 0 outside $Y$, due to an approximation by continuous functions, one has for every finite Borel measure $\nu$ on $X$
\begin{equation}
(\cL^*(\nu))(Y))=\int_X \cL(I_Y) \,d\nu =\int_{f(Y)}\exp\phi\circ f|_Y^{-1}\,d\nu.
\end{equation}
Hence the (positive) eigen-measure $m_\phi$ has Jacobian for $(f|_Y)^{-1}$ equal to
$\exp(\phi\circ f|_Y^{-1})/\lambda$, hence $f$ has Jacobian
$\exp(-\phi)$ multiplied by an eigenvalue $\lambda:=\exp P(\phi)$.

The proof of the existence of an invariant Gibbs measure equivalent to $m_\phi$ is harder.
One first proves the existence of a positive eigenfunction $u_\phi$ for $\cL$ and then
defines $\mu_\phi=u_\phi m_\phi$.
For a more complete introduction to this theory, see e.g.~\cite{PUbook}.

\section{Introduction to dimension 1}\label{dim1}

Thermodynamic formalism is useful for studying properties of the underlying space $X$. %, depending on which potential function $\phi$ is used.
In dimension 1, for $f$ real of class $C^{1+\e}$ or $f$ holomorphic, for an expanding repeller $X$, considering $\phi=\phi_t:=-t\log |f'|$ for $t\in\mathbb{R}$, \eqref{Gibbs-eq} gives
\begin{align}%\label{diam}
\mu_{\phi_t}(f_x^{-n}(B(f^n(x),r_0)))\approx \exp (S_n\phi(x)-nP(\phi_t))\approx \label{diam}
\\
\diam f_x^{-n}(B(f^n(x),r_0))^{t}\exp -nP(\phi_t). \nonumber
\end{align}
The latter follows from a comparison of the diameter with the inverse of the absolute value of the derivative of $f^n$ at $x$, due to  {\it bounded distortion}. Here, the symbol ``$\approx$'' denotes that the mutual ratios are bounded by a constant.

\smallskip

When $t=t_0$ is a zero of the function $t\mapsto P(\phi_t)$, this gives
\begin{equation}\label{Bowen}
\mu_{\phi_{t_0}} (B)\approx (\diam B)^{t_0}
\end{equation}
for all small balls $B$ (the $t_0$-Ahlfors measure property).
We obtain the so-called Bowen's formula for Hausdorff dimension:
\begin{equation}\label{HDformula}
\HD(X)=t_0.
\end{equation}
Moreover, the Hausdorff measure of $X$ in this dimension is finite and nonzero.
%of measure $\mu_{\phi_{t_0}}
%$\HD{\mu_{phi_{t_0}} =h_{\mu_{\phi_{t_0}}}/\chi(\mu_{phi_t})$.

\smallskip

A model example of application is the proof of
\begin{theorem}\label{z2+c}
For $f_c(z):= z^2+c$ for  an arbitrary complex number $c\not=0$ sufficiently close to 0,
the invariant Jordan curve $J$ (Julia set for $f_c$) is a fractal, i.e.~has Hausdorff dimension bigger than 1.
\end{theorem}

\begin{proof}[Sketch of Proof] $t_0>1$ yields $\HD(J)=t_0>1$ by \eqref{Bowen} (one does not need to use the invariance of $\mu_{\phi_{t_0}}$). %; the  conformal measure $m_\phi$ also satisfying \eqref{diam} is sufficient).

The case $t_0=1$ yields by \eqref{Bowen} finite Hausdorff measure in dimension 1, i.e.~the rectifiability of $J$.
To conclude that $J$ is a circle and $c=0$, one can use ergodic invariant measures in the classes of harmonic ones on $J$ from inside and outside. They  must coincide.
  This relies on Birkhoff's Ergodic Theorem, the heart of ergodic theory.  This is an ``echo'' of the celebrated Mostov Rigidity Theorem. See \cite{Sullivan:82} and \cite[Theorem 9.5.5]{PUbook}.
\end{proof}

\bigskip

In dimension 1 (real or complex), we call $c$ a critical point if the derivative $f'(c)=0$. The set of critical points will be denoted by $\Crit(f)$.

\smallskip

In this survey, we allow for the presence of critical points and concentrate mainly on two cases:

\smallskip

1. (Complex case) $f$ is a rational mapping of degree at least 2 of the Riemann sphere $\ov{\C}$. We consider $f$ acting on its Julia set $K=J(f)$.

\smallskip

For entire or meromorphic maps see e.g.~\cite{BKZ1, BKZ2}, compare Definition \ref{hypdim}.

\smallskip

2. (Real case) $f$ is a {\it generalized multimodal map} defined on a neighbourhood $U_K\subset\bR$ of its compact  invariant subset $K$. We assume that $f\in C^2$,
%(we can assume $C^3$ and the lack of indifferent periodic points, as then $f$ can bereplaced outside $K$ to have bounded distortion; however for generality we want to allow indifferent periodic points).
is non-flat at all of its turning and inflection critical points, satisfies the  {\it bounded distortion} property for iterates, abbr. BD, see \cite{PrzRiv:14},  is topologically transitive and has positive
topological entropy on $K$.

We assume that $K$ is a maximal invariant subset of a finite union of pairwise disjoint closed intervals $\hat I=I^1\cup \dots \cup I^k \subset U_K$ whose endpoints are in $K$. (This maximality corresponds to the Darboux property, compare \cite[Appendix A]{PrzRiv:14} and \cite[page 49]{MiSzlenk}.)
%, see \cite{MisSzl:80}.)
We write $(f,K)\in {\cA}^{\BD}_+$, with the subscript + to mark positive entropy. In place of BD one can assume $C^3$ (and write $(f,K)\in {\cA}^3_+$), and assume that all periodic orbits in $K$ are hyperbolic repelling. Indeed, changing  $f$ outside $K$ if necessary, one can get the corrected $(f,K)$ in ${\cA}^{\BD}_+$.%; denote the related class by ${\cA}^{3}_+$.)

\smallskip

 Recall the notions concerning periodic orbits: {\it Parabolic} means $f^n(p)=p$ with $(f^n)'(p)$ being a root of unity. For $|(f^n)'(p)|=1$ the term {\it indifferent periodic} is used and for $|(f^n)'(p)|>1$ the term {\it hyperbolic repelling}. If $|(f^n)'(p)|<1$ the orbit is called {\it hyperbolic attracting}.

\smallskip

For the real setting, see \cite{PrzRiv:14}, \cite{GPR2} and \cite{Prz:16}. Examples are provided by basic sets in the spectral decomposition \cite{dMvS}.

\smallskip
{\bf Question.} Are there any other examples?

\smallskip

{\bf Problem}. Generalize the real case theory, see further sections, to the piecewise continuous maps, that is allow the intervals $I^j$ to have common ends (see \cite{HU} for some results in this direction).

\smallskip

 In this survey, we compare equilibrium states to (refined) Hausdorff measures in the complex case. For the real case, we refer the reader to \cite{HKe}  and the references therein.

\section{Hyperbolic potentials}\label{Hyperbolic potential}

For general $f:X\to X$ and $\phi:X\to \mathbb{R}$ as in Definition \ref{top_pres} the following conditions
are of special interest \cite{InoRiv:12},

1) $P(f,\phi)>\sup\phi$,

2) $P(f^n,S_n\phi)>\sup_{X} S_n\phi$
for an integer $n$,

3) $P(f,\phi)>\sup_{\nu\in\cM(f)}\int\phi\, d\nu$,

4) For each equilibrium state $\mu$ for the potential $\phi$, the entropy $h_\mu(f)$ is positive. %its Lyapunov exponent $\chi(\mu):=\int\log|f'|\,d\mu$ is positive. That is for $\mu$-a.e.~$x$

 \smallskip

 The conditions 2) -- 4) are equivalent, see \cite[Proposition 3.1]{InoRiv:12}.
 Potentials $\phi$ satisfying them have been called in \cite{InoRiv:12} {\it hyperbolic}.
 The condition 1) has longer traditions, see \cite{DenUrb:91}.
 The intuitive meaning is that no minority of trajectories carries the full pressure.
 %If $f$ is distance expanding and has positive topological entropy then every potential $\phi$ is hyperbolic

 For every $f:\ov\C\to\ov\C$ rational of degree at least 2 and $\phi:J(f)\to\mathbb{R}$ H\"older continuous,
 the
 following condition is also equivalent to 2) -- 4), see \cite{InoRiv:12}:

 \smallskip

 5) For each ergodic equilibrium state $\mu$ for $\phi$, the Lyapunov exponent $\chi(\mu):=\int \log |f'|\,d\mu$ is positive, that is for $\mu$-a.e. $x$, \
$\chi(\mu)=\chi(x):=\lim_{n\to\infty}\frac1n\log |(f^n)'(x)|>0.$

\smallskip

  The conditions 2)-5) are also equivalent in the real case for $(f,K)\in\sA^{\BD}_+$ or $(f,K)\in\sA^3_+$ and all periodic orbits hyperbolic repelling. The arguments in \cite{InoRiv:12} work. See also \cite{RivLi2}.

\smallskip

\begin{theorem}\label{equi-hyp}
Let $f:\ov\C\to\ov\C$ be a rational mapping as above. If $\phi$ is a H\"older continuous hyperbolic potential on $J(f)$, then there exists a unique equilibrium state $\mu_\phi$. For every H\"older $u:J(f)\to\mathbb{R}$, the Central Limit Theorem (abbr. CLT) for the sequence of random variables $u\circ f^n$ and $\mu_\phi$ holds.
\end{theorem}
For a proof, see \cite{Prz:90} and preceding \cite{DenUrb:91}.
To find this equilibrium one can iterate the transfer operator proving
$\cL^n(1\!\!1)/\exp nP(f,\phi) \to u_\phi$. The convergence is uniformly $\exp-\sqrt{n}$ fast and the limit  is H\"older continuous, \cite{DPU}. Finally, define $\mu_\phi:=u_\phi\cdot m_\phi$, as at the end of
Section \ref{Introduction}.

\begin{remark}\label{inducing}
Given $\mu_\phi$ a priori, an efficient way
to study it is an inducing method, see \cite{SzoUrbZdu:15}, i.e.~the use of a return map $A\ni x\mapsto f^{n(x)}(x)\in A$  for $A$ and $n(x)$ adequate to $\mu_\phi$. Then one proves even an exponential convergence
(with any $u$ H\"older in place of $1\!\!1$),
which yields exponential mixing, hence stochastic laws for $u\circ f^n$ for H\"older $u$,
%for $\mu_\phi$,
e.g.~CLT, LIL, compare Sections \ref{LIL} and \ref{sec:acc}. See also Remark \ref{ind2}. The key feature is the exponential decay of $\mu_\phi(A_n)$, where
$A_n:=\{x\in A: n(x)\ge n\}$.

\smallskip

See  also  \cite{BT2} for the real case, and stronger \cite{RivLi1} and \cite{RivLi2} including also the complex case proving the exponential convergence to $u_\phi$, hence CLT and LIL.
See also \cite{SzoUrbZdu:14} for endomorphisms $f$ of higher dimensional complex projective space, where 1) is replaced by a  stronger ``gap'' assumption.
\end{remark}

\

%Real case Bruin, Todd ??? %, Iommi, Pesin ?

\section{Non-uniform hyperbolicity in real and complex dimension 1} %weaker versions,1D real, complex, CE and various equiv. conditions. Invariant sets

Here we discuss a set of conditions, valid in both the real and complex situations. Below we concentrate on the case of complex rational maps  with $K=J(f)$, only remarking differences in the real case.
\smallskip

\noindent
(a) \ CE. \  {\it Collet-Eckmann condition.}\
	There exist $\la_{CE} > 1$ and $C>0$ such that for every critical point $c$ in $J(f)$, whose forward orbit does not meet other critical points,  for every $n \ge 0$ we have
$$
	|(f^n)'(f(c))|\ge C \la_{CE}^n. %\eqno (CE \ \hbox{at} \ c).
$$
Moreover, there are no parabolic (indifferent) periodic orbits.

\smallskip

\noindent
 (b) \ CE2$(z_0)$. \ {\it Backward} or {\it second Collet-Eckmann condition at $z_0 \in J(f)$.}
	There exist $\la_{CE2}>1$ and $C>0$ such that for every $n \ge 1$ and
every $w \in f^{-n}(z_0)$ (in a neighbourhood of $K$ in the real case)
$$
|(f^n)'(w)|\ge C \la_{CE2}^n.
$$
\smallskip
\noindent
(b') \ CE2. \ {\it The second Collet-Eckmann condition.} $\CE2(c)$ holds for all critical points $c$ not in the forward orbit of any other critical point.

\smallskip

\noindent
(c) \ TCE. \  {\it Topological Collet-Eckmann condition.}
        There exist  $M \ge 0, P \ge 1$, $r>0$ such that for every
$x\in K$ there exists a strictly increasing sequence of positive integers
$n_j$,  $j=1,2,\dots$, such that  $n_j \le P \cdot j$ and for each $j$ (and discs $B(\cdot)$ below understood in $\ov\C$ or $\mathbb{R}$)
\begin{equation}\label{TCE}
        \#\{0 \le i < n_j:   \Comp_{f^i(x)}f^{-(n_j-i)}B(f^{n_j}(x),r))
                \cap\Crit(f) \not= \emptyset \} \le M,
\end{equation}
where in general $\Comp_z V$ means for $z\in V$ the component of $V$ containing $z$.

\smallskip

In the real case, one adds the condition that there are no parabolic periodic orbits, which is automatically true in the case of  complex rational maps.

\smallskip
\noindent
(d) \ ExpShrink. \ {\it Exponential shrinking of components.}
        There exist $\lambda_{\Exp}>1$ and $r>0$  such that for every $x \in K$,
every $n > 0$ and every connected component $W_n$ of $f^{-n}(B(x, r))$ for the disc (interval) $B(x,r)$ in $\ov\C$ (or $\mathbb{R}$), intersecting $K$ %(pull-back, see \cref{pull-back}), we have
\begin{equation}\label{ExpShrink}
        \diam (W_n) \le \lambda_{\Exp}^{-n}.
\end{equation}

%\noindent
%(e) \ H\"older Coding Tree.
%	There are constants $\la_{\Ho} > 0$ and $C > 0$ and a point $w_0 \in \ov\C$  such that the following holds.
	%For each preimage $w_1 \in f^{-1}(w_0)$ there exist a continuous path
        %$\gamma_{w_1}\subset \ov\C \setminus \cup_{n \ge 1} f^n(\Crit)$, without
%self-intersections, joining $w_0$ to $w_1$,
%such that for every $n \ge 0$ and every connected component $\gamma$ of $f^{-n}(\gamma_{w_1})$ we have $\diam(\gamma) \le C \la^{-n}_{\Ho}$.

\noindent
(e) LyapHyp ({\it Lyapunov hyperbolicity}). %{\it  Lyapunov exponents of invariant measures are bounded away from zero}.
        There is a constant $\lambda_{\Lyap} > 1$ such that the Lyapunov exponent $\chi(\mu)$
of any ergodic measure $\mu\in\cM(f,K)$
        %invariant probability measure $\mu$
        %supported on $K$
        satisfies $\chi(\mu)\ge \log \la_{\Lyap}$.

\smallskip

%\noindent
%(f) \ Negative Pressure. %{\it Pressure for large $t$ is negative}.
 %       For large values of $t$ the pressure function      $P(t)$ is negative.
%\smallskip

 \noindent
(f) \ UHP. \ {\it Uniform Hyperbolicity on periodic orbits.}
        There exists $\lambda_{\Per} > 1$ such that every periodic point $p
\in K$ of period $k \ge 1$ satisfies
$$
        |(f^k)'(p)|\ge \lambda_{\Per}^k.
$$

%\noindent
%$\bullet$ \ UHPR. \ {\it Uniform Hyperbolicity on hyperbolic repelling periodic orbits.}
%The same as UHP but only for hyperbolic repelling periodic points in $K$.

\

We distinguish LyapHyp as the most adequate among these conditions to carry the name (strong) non-uniform hyperbolicity.\footnote{Then all H\"older continuous potentials are hyperbolic, see  Condition 5) in Section \ref{Hyperbolic potential} and \cite{InoRiv:12}.}

%In the real case one adds: for $w$ close to $K$, as $K$ need not be backward invariant

%\emph{close to $K$} means that for  a constant $R>0$ (not depending on $n$ and $w$) the pull-back of $B(z_0,R)$ for $f^{n}$ containing  $w$ intersects $K$.

\begin{theorem}\label{Nonunifhyp}
1. The conditions (c)--(f) and else (b) for some $z_0$
%(with $CE2(z_0)$ for some $z_0)$)
are equivalent in the complex case. In the real case, the equivalence also holds under the assumption of weak isolation (see the definition below).

2. In the complex case, the suprema over all possible constants $\lambda_{\Exp}$,
$\lambda_{CE2}$ (supremum over all $z_0$),
 $\lambda_{\Per}$ and $\lambda_{\Lyap}$ coincide.

 3. Both CE and CE2 imply (c)--(f).

 4. If there is only one critical point in the Julia set in the complex case or if $f$ is $S$-unimodal on $K=I$ in the real case, i.e.~has just one turning critical point $c$ and negative Schwarzian derivative on $I\setminus \{c\}$,
	then all
	conditions above are equivalent to each other.

\end{theorem}
For more details, see \cite{PrzRivSmi:03}, \cite{Riv:12} and \cite{PrzRiv:14}.

\begin{definition}\label{weak-isolation}
 $(f,K)\in\cA$ is said to be {\it weakly isolated}
if there exists
 an open neighbourhood $U$ of $K$ in the domain of $f$
such that for every  $f$-periodic orbit $O(p)\subset U$ is contained in $K$.
\end{definition}

In the complex case, we can replace \eqref{TCE} by
$$
\deg \Bigl (f^{n_j} \bigl |_{\Comp_x f^{-n_j}(B(f^{n_j}(x),r))   } \Bigr )
 \le  M'
 $$
 for a constant $M'$. In the real case, this condition is weaker than \eqref{TCE} since $f$ mapping $W_{n+1}$ into $W_n$ may happen not surjective. It can have folds, thus
 truncating backward trajectories of critical points acquired before when pulling back.

\smallskip

 In the real case, the proof of CE$\Rightarrow$TCE can be found in \cite{NP}. For the complex case, we refer the reader to  \cite{PrzRoh1}.

 \smallskip

The implication TCE$\Rightarrow$CE was proved in the complex case in \cite[Theorem 4.1]{P-Holder}. The proof used the idea of the ``reversed telescope'' by \cite{GraSmi1}.
In the real case, this implication was proved  for $S$-unimodal maps  in \cite{NowSands}.
In presence of more than one critical point this implication may be false, see \cite[Appendix C]{PrzRivSmi:03}.

%The negative Schwarzian derivative can be omitted due to e.g. ?????

\smallskip

{\bf Question.} Is this implication true for every $(f,K)\in\cA_+^{\BD}$ with one critical point, provided it is weakly isolated? See Definition \ref{weak-isolation}. It seems that the answer is yes.

\smallskip

Since the condition TCE is stated in purely topological terms (in the class of maps without indifferent periodic orbits), it is invariant under topological conjugacy.
	So we obtain the following immediate corollary.

\begin{corollary}
	All equivalent conditions listed above are invariant under topological conjugacies between $(f,K)$'s).
	% So are conditions \ CE and \ CE2(Crit) in the case where there is only onecritical point in the Julia set or in $K$ in the real case.
\end{corollary}

Another proof of the topological invariance of CE in the complex case was provided in \cite{PrzRoh2} with the use of Heinonen and Koskela criterion for quasi-conformality, \cite{HeiKos:95}.

Note that this topological invariance is surprising, as all the conditions except TCE are expressed in geometric-differential terms. I do not know how to express CE for unimodal maps of interval in the (topological-combinatorial)
kneading sequence terms.

\smallskip

An important lemma used here has been an estimate of an average distance in the logarithmic scale of every orbit from $\Crit(f)$, see \cite{DPU}. Namely
\begin{lemma}\label{av_dist}
\begin{equation}\label{average distance}
	{\sum_{j=0}^n}{}'  -\log |f^j(x)-c| \le Qn
\end{equation}
for a constant $Q>0$ every $c\in\Crit(f)$, every $x\in K$ and every integer
$n>0$. $\Sigma'$ means that we omit in the sum an index $j$ of smallest distance  $|f^j(x)-c|$.
\end{lemma}

%Condition \eqref{av_dist} is often used for $x=f(c)$ in the unimodal case  with $Q$ arbitrarily close to 0, which holds however for "most" but not for all maps, see M. Tsujii ???

An order of proving the equivalences in Theorem \ref{Nonunifhyp} is

\noindent CE2$(z_0)\Rightarrow$ExpShrink$\Rightarrow$LyapHyp$\Rightarrow$UHP$\Rightarrow$CE2$(z_0)$ and separately

\noindent ExpShrink$\Leftrightarrow$TCE.
E.g. assumed UHP one proves CE2$(z_0)$ by ``shadowing'', compare the beginning
of Section \ref{sec:other}.

\section{Geometric pressure and equilibrium states}  %phase transitions (Rivera-Letelier, PRS, LPS, Spanning, GPR), Lyapunov exponents

We go back to topological pressure, Definition \ref{top_pres}, but for $\phi=-t\log |f'|$, $t\in \mathbb{R}$ in the complex $K=J(f)$ or real cases, where $\phi$ can attain the values $\pm\infty$ at the critical points of $f$. See the beginning of Section \ref{dim1}. We call it the geometric pressure, because it is useful in studying of geometry of the underlying space, e.g.~as in \eqref{HDformula} via equilibrium states for all $t$.

The definition of $P_{\var}(f, -t\log|f'|)$ in  Definition \ref{top_pres} makes sense due to $\chi(\mu)\ge 0$ for all $\mu\in\cM(f)$, in particular due to the integrability of $\log |f'|$, see
 \cite{Prz:93} and \cite[Appendix A]{Riv:12} for a simpler proof. We conclude that it is convex and monotone decreasing. We start by defining a quantity  occurring equal to $P(t)=P_{\var}(t):=P_{\var}(f,-t\log|f'|)$, to explain its geometric meaning, compare with Section \ref{dim1}.

 \begin{definition}[Hyperbolic pressure]\label{hyperbolic pressure}
$$
P_{\hyp}(t):= \sup_{X\in\cH(f,K)} P(f|_X,-t\log|f'|) ,
$$
where $\cH(f,K)$ is defined as the space of all compact forward $f$-invariant (that is $f(X)\subset X$)
hyperbolic subsets of $K$, repellers in $\mathbb{R}$.
\end{definition}

From this definition, it immediately follows that:

%\noindent \cite[Corollary 12.5.12]{PU} in the complex case) the following

\begin{proposition}\label{hypdim}{\rm (Generalized Bowen's formula, compare \eqref{HDformula})}
The first zero $t_0$ of $t\mapsto P_{\hyp}(K,t)$ is equal
to the hyperbolic dimension $\HD_{\hyp} (K)$ of $K$, defined by
$
\HD_{\hyp} (K):=\sup_{X\in\cH(f,K)} \HD(X).
$
%where the supremum is taken over all compact forward $f$-invariant
%hyperbolic subsets of $K$, repellers in $\mathbb{R}$.
\end{proposition}

For the discussion $\HD_{\hyp}(J(f))$ vs $\HD(J(f))$, see \cite[Section 2.13]{Lyubich_ICM14}.

\medskip

Below we state Theorem \ref{equi2} proved in \cite{PrzRiv:11} in the complex setting and in
\cite{PrzRiv:14} in the real setting.
It extends \cite{BT1, PS} and \cite{IT}. See also impressive \cite{DT}.

\begin{figure}
\begin{minipage}[c]{\linewidth}
\centering
\begin{overpic}[scale=.35]{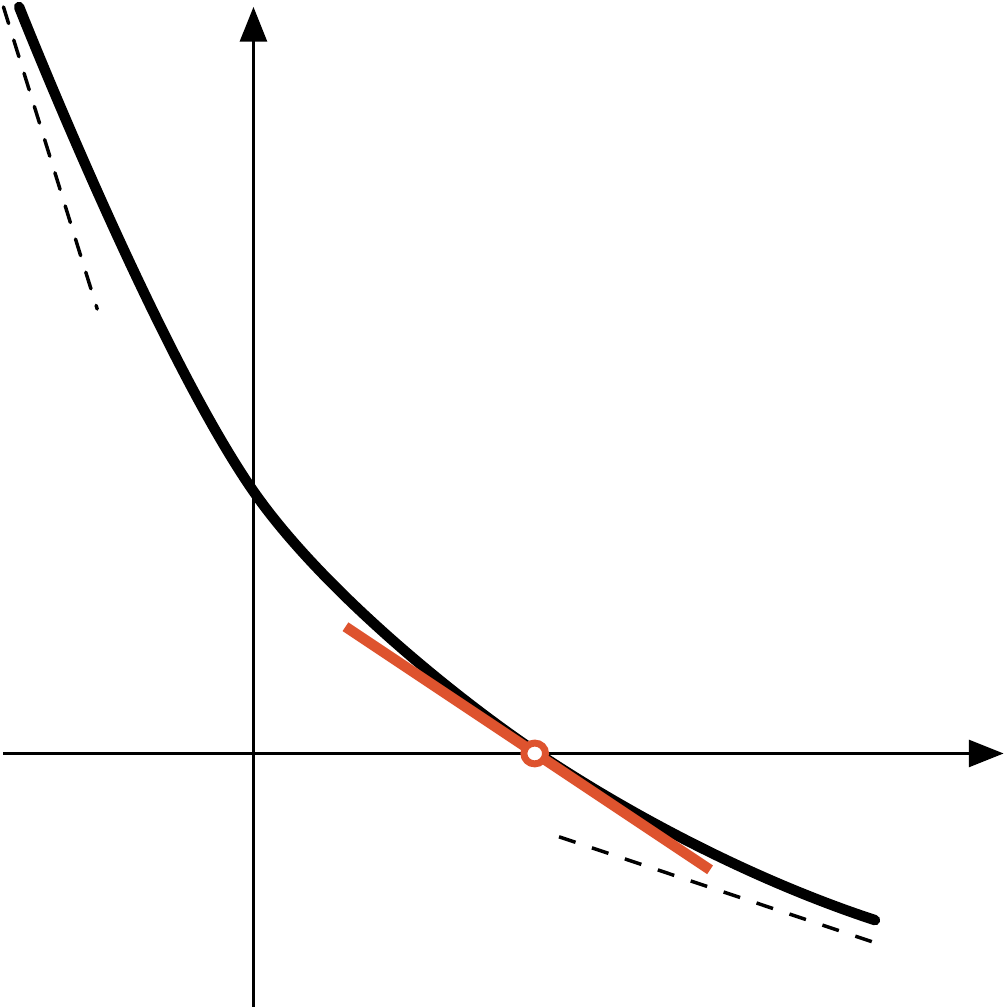}
        \put(102,23){\small$t$}
        \put(52,32){\small$t_0$}
        \put(30,93){\small$P(t)$}
        \put(-15,63){\tiny$-\chi_{\rm sup}$}
       % \put(26,28){\tiny$-\chi^\ast$}
        \put(50,4){\tiny$-\chi_{\rm inf}$}
 \end{overpic}
 \hspace{0.5cm}
 \begin{overpic}[scale=.35]{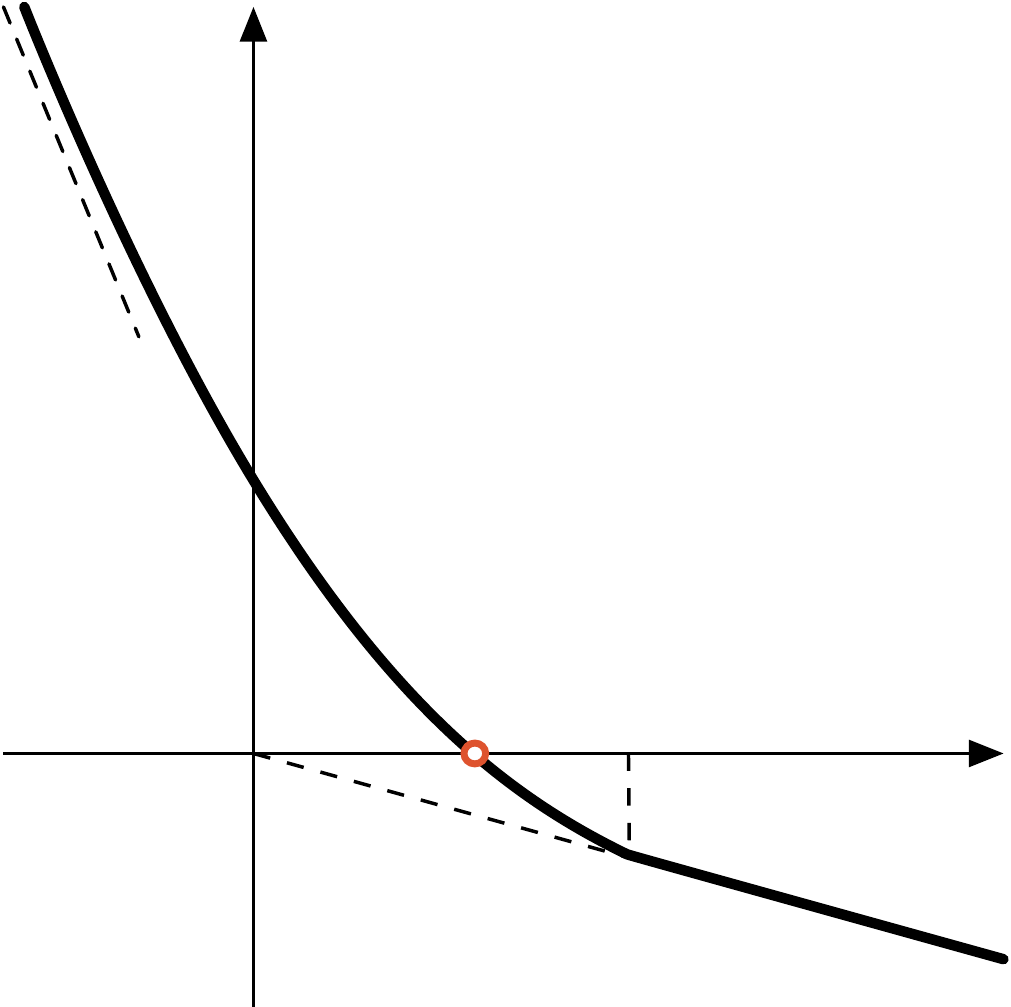}
        \put(102,23){\small$t$}
        \put(44,32){\small$t_0$}
        \put(60,32){\small$t_+$}
        \put(30,93){\small$P(t)$}
        \put(-15,63){\tiny$-\chi_{\rm sup}$}
        \put(60,4){\tiny$-\chi_{\rm inf}$}
\end{overpic}
 \hspace{0.5cm}
\begin{overpic}[scale=.35]{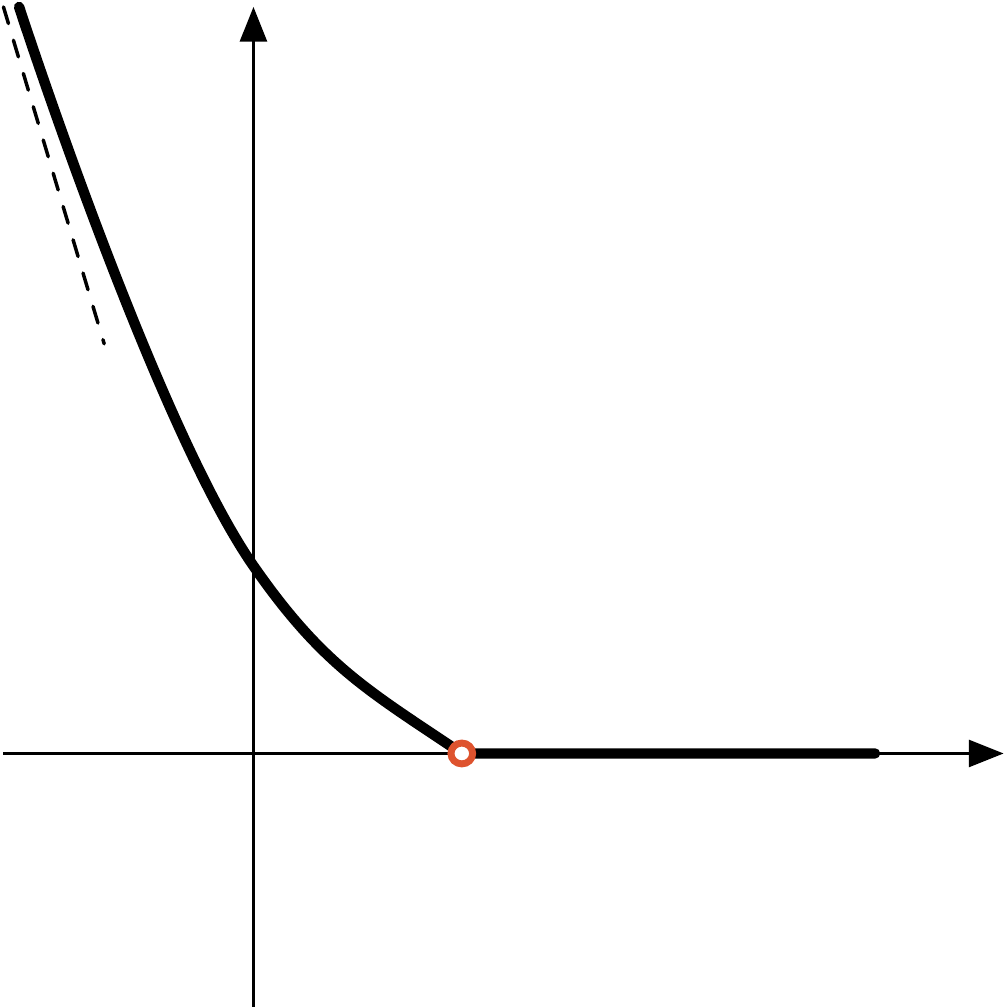}
        \put(102,23){\small$t$}
        \put(44,32){\small$t_0=t_+$}
        \put(30,93){\small$P(t)$}
        \put(-15,63){\tiny$-\chi_{\rm sup}$}
        \put(50,15){\tiny$-\chi_{\rm inf}$}
       % \put(50,3){\tiny$=-\chi^\ast=0$}
\end{overpic}
\caption{The geometric pressure: LyapHyp with $t_+=\infty$, \newline LyapHyp with $t_+<\infty$, and non-LyapHyp. This Figure is taken from \cite{GPR2}, see notation in Remarks below.}
\label{geometric-pressure}
\end{minipage}
\end{figure}
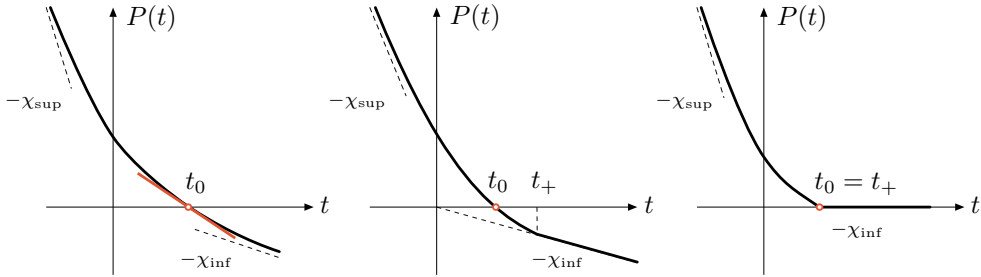

%The following holds, partially analogous to \cref{equi}.

%The works of M. Urba\'nski, B. Stratmann, H. Bruin, M. Todd, G. Iommi, Y. Pesin, S. Senti, P, J. Rivera-Letelier, ...

\begin{theorem}\label{equi2}

1. Real case, \cite{PrzRiv:14}. Let
$(f,K)\in {\sA}_+^3$
and let all $f$-periodic orbits in $K$ be hyperbolic repelling.
Then $P(t)$ is real analytic on the open interval bounded by the ``phase transition parameters'' $t_-$ and  $t_+$. For every $t\in (t_-,t_+)$,
%for the potential~$\phi_t$,
%$(t_-,t_+)$, i.e. the maximal possible domain. Precisely in this
the domain where
\begin{equation}\label{hyplog}
P(t)>\sup_{\nu\in\cM(f)} -t\int\log|f'|\, d\nu,
\end{equation}
there is a unique invariant  equilibrium state. It is ergodic and absolutely continuous with respect to an adequate conformal measure $m_{\phi_t}$ with the density bounded from below by a positive constant almost everywhere.
If furthermore $f$ is topologically exact on $K$ (that is for every $V$ an open subset of $K$ there exists $n\ge 0$ such that $f^n(V)=K$), then this measure is mixing, has exponential decay of correlations and it satisfies the Central Limit Theorem for Lipschitz gauge functions.

\smallskip

2. Complex case, \cite{PrzRiv:11}. The assertion is the same. One assumes a very weak expansion:  the existence of arbitrarily small nice, or pleasant, couples and hyperbolicity away from critical points.

\end{theorem}

{\bf Remarks.}
1) $t_-$ and $t_+$ are called the phase transition parameters. Since $P(0)=h_{\rm top}(f)>0$,  $t_-<0<t_+$, they need not exist; we say then they are equal to $-\infty$ and/or $+\infty$ respectively.  $P(t)$ is linear to the left of $t_-$ and to the right of $t_+$, equal to $t\mapsto -t\chi_{\sup}$ where $\chi_{\sup}:=\sup_\nu \chi(\nu)$ and
$t\mapsto -t\chi_{\inf}$, where $\chi_{\inf}:=\inf_\nu\chi(\nu)$, respectively.  Of course, $P(t)$ is not real-analytic at finite $t_-$ and $t_+$.

 2) For $f(z)=z^2-2$, $f:[-2,2]\to[-2,2]$ (the Tchebyshev polynomial), we have $f(2)=2, f'(2)=4,
\chi(l)=\log 2$, where $l$ is the normalized length measure. We have $P(t)=\log 2-t\log 2$ for $t\ge -1$  and $P(t)=-t\log 4$ for $t\le -1$, so $t_-=-1$, $P(t)$ is non-differentiable at $t_-$ and for $t=-1$ there are two ergodic equilibrium states: Dirac at $z=2$ and $l$.

3) For any $f$ non-LyapHyp, $t_+=t_0<\infty$. However  $t_+<\infty$ can happen even for $f$ LyapHyp, see \cite{MS2}   and \cite{CR1, CR2}.

4) Notice that the condition \eqref{hyplog} is similar to the condition 3) from Section \ref{Hyperbolic potential}. For $f$ LyapHyp and $t>t_+$, no equilibrium state can exist, see \cite{InoRiv:12}.
%The reason is that otherwise pressure should be supported by few trajectories, whereas positive topological entropy yields an abundance of them.

5) For real $f$ as in Theorem \ref{equi2} satisfying LyapHyp and $K=\hat I$, we have $t_0=1$ and for $-\log|f'|$ we conclude that
a unique equilibrium state exists which is a.c.i.m.( that is: invariant absolutely continuous with respect to Lebesgue measure). In fact this assertions hold even for $t=t_0=t_+=1$ with very weak hyperbolicity properties e.g.~$|(f^n)'(f(c))|\to\infty$ for all $c\in\Crit(f)$, see
\cite{BRSvS} and \cite{ShenStrien}. For the complex case, see \cite{GraSmi:09} and stronger \cite{RivShen:14}.

\smallskip

 \begin{remark}\label{ind2}
 In the proof of Theorem \ref{equi2},
 %, similarly to the
 we use (compare with the Remark \ref{inducing}) a return map $F(x)=f^{n(x)}$ to a ``nice'' (Markov) domain. However unlike in \cite{SzoUrbZdu:15}, we do not use in the construction of this set the equilibrium measure $\mu_\phi$ because we do not know a priori that it exists. The construction is geometric. $F$ is an infinite Iterated Function System, more precisely the family of all branches of $F^{-1}$ is,
 see \cite{MU} and \cite{Pesin} and references therein,
 expanding due to the ``acceleration'' from $f$ to $F$. Then we consider an equilibrium state ${\tt P}$ for $(F,\Phi)$ where $\Phi(x):=\sum_{j=0}^{n(x)-1}\phi_t(f^j(x))$, and consider an equivalent conformal measure.
We propagate these measures to the Lai-Sang Young tower $\{(x,j): 0\le j<n(x)\}$ and
project by $(x,j)\mapsto f^j(x)$ to $K$.\footnote{For applications to decide the existence or nonexistence of a finite a.c.i.m. for maps of interval with flat critical points or for entire or meromorphic maps depending on the {\tt P}-integrability of the first return time, see papers by N.~Dobbs, B.~Skorulski, J.~Kotus, G.~\'Swi\c atek.}

 Stochastic properties of ${\tt P}$ stay preserved along the construction to $\mu_\phi$. The analyticity of $P(t)$ follows from expressing $P(t)$ as
zero of a pressure for $F$ with potential depending on two parameters and Implicit Function Theorem. The latter idea came from \cite{StratUrb}.
\end{remark}

\begin{remark}
For probability measures $\mu_n$ weakly* convergent to some $\hat\mu$, in presence of critical points
$\int\log |f'|\,d\mu_n$ need not converge to  $\int\log |f'|\,d\hat\mu$. Only upper semicontinuity holds. Therefore, for $t>0$, the equilibrium states for $t_n\to t$ need not converge to an equilibrium state for $t$. A priori, the free energy in the Definition \ref{top_pres} can jump down. However, a modification of this method to prove existence of equilibria works, see \cite{DT}.

Notice also that passing to a weak*-limit with averages of Dirac measures on $\{x, \dots, f^n(x)\}$ proves  %\mu_n=\frac1n\Deltaproves ,
$\limsup_{n\to\infty}\sup_{x\in K}\frac1nS_n(\log|f'|)(x) \le \chi_{\max}$. However an analogous inequality $\liminf \dots \ge\chi_{\inf}$ is obviously false. These observations contribute to the understanding of Lyapunov spectrum.
\end{remark}

\smallskip

{\bf Remarks on the Lyapunov spectrum.}  Theorem \ref{equi2} allows us to express the so-called dimension spectrum for Lyapunov exponents
with the use of Legendre transform, that is for all $\a>0$ and $\cL(\a):=\{x\in K: \chi(x)=\a\}$
\begin{equation}\label{Lyap_spec}
\HD(\cL(\a))=
%F(\alpha) \eqdef
\frac{1}{|\alpha|} \inf_{t\in\bR}
\left(P(t)+\alpha t \right).
\end{equation}
An ingredient is {\bf Ma\~n\'e's equality}
\begin{equation}\label{mane}
\HD(\mu)=h_\mu(f)/\chi(\mu)
\end{equation}
provided $\chi(\mu)>0$, \cite{PUbook}, where $\HD(\mu):=\sup\{\HD(X): \mu(X)=1\}$, applied to $\mu_{\phi_t}$.

The equality
 \eqref{Lyap_spec} concerns regular $x$'s, where $\chi(x)=\lim_{n\to\infty} \frac1n\log |(f^n)'(x)|$
 exists. It is  also possible to provide formulas or at least estimates for Hausdorff dimension of
the sets of irregular points $\cL(\a,\b):=\{x\in K: \underline{\chi}(x)=\a, \overline{\chi}(x)=\b\}$
for lower and upper Lyapunov exponents where we replace lim by $\liminf$ and  $\limsup$ respectively.
See \cite{GPR1} and \cite{GPR2} for this theory in complex and real settings.

However, these papers give no information about the size of sets with zero (upper) Lyapunov exponent.
Note at least that if $J(f)\not=\ov\C$ then ${\rm Leb}_2\{x\in J(f):\ov\chi(x)>0\}=0$. This is so because $\ov\chi(x)>0$ implies there exists $\sN\subset \Z_+$ of positive upper density, such that for $n\in\sN$, \eqref{ExpShrink} and \eqref{TCE} hold, see \cite[Section 3]{LPS}.

We do not know whether $\chi(x)=-\infty$ can happen for $x$ not pre-critical, except there is only one critical point in $K$, where $\chi(x)>-\infty$ follows from \eqref{average distance},  see \cite[Lemma 6]{GPR1}.

For $x$ being a critical value we can prove (in analogy to $\chi(\mu)\ge 0$):

\begin{theorem}[\cite{LPS}]\label{LPShen}
If for a rational function $f:\ov\C\to\ov\C$ there is only one critical point $c$ in $J(f)$  and no parabolic periodic orbits, then $\underline\chi(f(c))\ge 0$.
\end{theorem}
For $S$-unimodal maps of interval this was proved
%much earlier
by \cite{NowSands}.

\section{Other definitions of geometric pressure}\label{sec:other}

%Useful is the notion of

\begin{definition}[safe]\label{safe} See \cite[Definition 12.5.7]{PUbook}. We call $z\in K$
\textit{safe} if $z\notin \bigcup_{j=1}^\infty(f^j(\Crit(f)))$ and for every $\delta>0$ and all $n$ large enough
$B(z, \exp (-\delta n))\cap \bigcup_{j=1}^n(f^j(\Crit(f)))=\emptyset$.
\end{definition}
Notice that this definition implies that all points except at most a set of Hausdorff dimension 0, are safe.

\begin{definition}[Tree pressure]\label{treep}
For every $z\in K$ and $t\in \mathbb{R}$ define
\begin{equation}\label{treep-formula}
P_{\tree}(z,t)=\limsup_{n\to\infty}\frac1n\log\sum_{f^n(x)=z,\, x\in K} |(f^n)'(x)|^{-t}.
\end{equation}
\end{definition}
Compare with $P(f,\phi)$ from Theorem \ref{Gibbs}. Under suitable conditions, e.g.~for $z$ ``safe''
%and expanding, as defined below) limsup can be replaced by liminf, i.e.
the limit exists, it is independent of $z$ and equal to $P(t)$. See \cite{Prz:99}, \cite{PrzRivSmi:03} and \cite{PUbook} for the complex case and \cite{PrzRiv:14} and \cite{Prz:16} for the real case.

A key is to extend all trajectories $T_n(x)=\{x,\dots,z\}$ backward and forward by time $m\ll n$ to get an Iterated Function System for $f^{n+m}$ and to consider its limit set. Its trajectories for time $n$ ``shadow'' $T_n(x)$. This proves $P_{\tree}(z,t)\le P_{\hyp}(t)$. The opposite inequality is immediate.

(Similarly one proves $P_{\var}(t)\le P_{\hyp}(t)$. Given $\mu$ with $\chi(\mu)>0$ one captures a hyperbolic $X$ by Pesin-Katok method.)

%\begin{proposition} See \cite{PrzRivSmi:04, PrzRiv:14, Prz:16}  For

\medskip

%Given $z$ safe we "go to large scale "$D=f^m(B(z,\exp(-\delta n$ for $m\ll n$ and consider preimages for $f^N$ of the components of $x

For a continuous potential $\phi:X\to\mathbb{R}$, consider
\begin{equation}\label{Psep}
P_{\rm {sep}}(f,\phi):=\lim_{\e\to 0}\limsup_{n\to\infty}\frac1n \log \big(\sup_Y \sum_{y\in Y}\exp S_n\phi(y)\big),
\end{equation}
where the supremum is taken over all $(n,\e)$-separated sets $Y\subset X$, that is such $Y$ that for every distinct $y_1,y_2\in Y$, $\rho_n(y_1,y_2)\ge \e$, where $\rho_n$ is the metric defined by $\rho_n(x,y)=\max\{\rho(f^j(x),f^j(y)): j=0,\dots,n\}$.

For $\phi=-t\log|f'|$ for positive $t$, in presence of critical points for $f$, $P_{\rm {sep}}$ is always equal to $\infty$ by putting a point of a separated set at a critical point.
So we replace it by the tree pressure. One can however use infimum over $(n,\e)$-spanning sets, thus defining \underline{$P_{\rm {spanning}}(f,\phi)$}. This is a valuable notion, often coinciding with other pressures. See \cite{Prz:16} for an
outline of a respective theory. Let me mention only that this is equal to $P(f,-t\log|f'|)$ for $t>0$  in the complex case if

\begin{definition}\label{wbls}
$f$ is {\it weakly backward Lyapunov stable} which means that for every $\delta>0$ and $\e>0$ for all $n$ large enough and every disc $B=B(x,\exp -\delta n)$ centered at $x\in K$, for every $0\le j \le n$ and every component $V$ of $f^{-j}(B)$ intersecting $K$, it holds that $\diam V\le\e$.
\end{definition}
This holds for all rational maps with at most one critical point whose forward trajectory is in $J(f)$ or is attracted to $J(f)$, due to Theorem \ref{LPShen}.

\smallskip

{\bf Question.} Does backward weak Lyapunov stability hold for all rational maps?

\smallskip

Finally, {\it periodic pressure} $P_{\Per}$ is defined as $P_{\tree}$ with $x\in \Per_n$ (periodic of period $n$)  rather than $f^n$-preimages of $z$. In \cite{PrzRivSmi:04}, this was proved for rational $f$ (see also \cite{BMS} for a class of polynomials) on $K=J(f)$ that $P_{\Per}(t)=P(t)$ provided

\smallskip

{\bf Hypothesis H}.\  For every $\d>0$ and all $n$ large enough, if for a set
$\cP\subset\Per_n$  for all
$p,q\in P$ and all $i:0\le i<n$ \
$\dist (f^i(p),f^i(q))< \exp -\d n$, then
$\# \cP\le \exp\d n$.

\smallskip

{\bf Question.}
  Does this condition always hold? In particular, can large bunches of periodic orbits exist with orbits exponentially close to a Cremer fixed point?

%\section{ On radial behavior of derivatives of univalent maps and geometric coding trees.  Harmonic and equilibrium measures vs Hausdorff measures.}

\section{Geometric coding trees, limit sets, Gibbs meets Hausdorff}\label{sec:gct}

The notion of geometric coding tree, g.c.t.,  already appeared in the work \cite{Ja}, where in the expanding case the finite-to-one property of the resulting coding was proved. It was used later
in \cite{Prz:85, Prz:86} and in a full strength in \cite{PUZ:89, PUZ:91} and papers following them. Similar graphs have since been constructed to analyse the topological aspects of
non-invertible dynamics, see for instance \cite{Nek, HaPi}.
%Firstly coding using a tree ${\sT}$ in the unit disc for say $g(z)=z^d$ and the transported  one

%+++++++++++++++++

%+++++++++++
%Question: which measures come from coding (subshift)

\begin{definition}\label{gct} %[\cite{PUZ:89}, \cite{PSkrzy}, \cite{PZ} and \cite{P1}.
Let $U$ be an open connected subset of the Riemann sphere $\ov\C$.
Consider a holomorphic mapping $f:U\to \ov\C$ such that $f(U)\supset U$ and
$f:U\to f(U)$ is a proper map.
Suppose that $\Crit(f)$ is finite.
 Consider an arbitrary $z\in f(U)$. Let $z^1,z^2,\dots,z^d$ be some
of the $f$-preimages of $z$ in $U$ with $d\ge 2$. Consider smooth
curves $\g^j:[0,1]\to f(U)$, \ $j=1,\dots,d$,   joining $z$ to $z^j$ respectively
(i.e.~$\g^j(0)=z, \g^j(1)=z^j$), intersections allowed, such that $\g^j\cap f^n(\Crit(f))=\emptyset$ for every $j$ and $n>0$.

%Let $\S^d:=\{1,\dots,d\}^{\Z_+}$ denote the one-sided shift space and $\varsigma$the shift to the left, compare \cref{Introduction}.
For every sequence
$\a=(\a_n)_{n=0}^\infty \in \S^d$ (shift space with left shift map $\varsigma$ defined
in Section \ref{Introduction})  define $\g_0(\a):=\g^{\a_0}$. Suppose that for some $n\ge
0$, for every $ 0\le m\le n$, and all $\a\in\S^d$,  curves $\g_m(\a): [0.1] \to U$ are already defined.
Suppose that for $1\le m\le n$ we have $f\circ \g_m(\a)=\g_{m-1}(\varsigma(\a))$,
and $\g_m(\a)(0)=\g_{m-1}(\a)(1)$.
Define the curves $\g_{n+1}(\a) $ so that the previous
equalities hold by taking respective $f$-preimages of curves $\g_n$.
For every  $\a\in\S^d$ and $n\ge 0$ denote $z_n(\a):=\g_n(\a)(1)$.

The graph $\sT={\sT}(z,\g^1,\dots,\g^d)$ with the vertices $z$ and $z_n(\a)$ and edges $\g_n(\a)$ is
called a {\it geometric coding tree} with the root at $z$. For every $\a\in\S^d$ the subgraph
composed of $z,z_n(\a)$ and $\g_n(\a)$ for all $n\ge 0$ is called an  {\it infinite geometric branch}
and denoted by $b(\a)$.
It is called {\it convergent} if the sequence $\g_n(\a)$
is convergent to a point in $\cl U$. We define the {\it coding map}
$z_\infty :{\sD}(z_\infty)\to \cl U$ by $z_\infty(\a):=\lim_{n\to\infty}z_n(\a)$ on the
domain ${\sD}={\sD}(z_\infty)$ of all such $\a$'s for which $b(\a)$ is convergent.

Denote  $\La:=z_\infty({\sD}(z_\infty))$. If the  map $f$ extends
holomorphically to a neighbourhood of its closure $\cl\Lambda$ in $\ov\C$,  then $\La$ is called a {\it quasi-repeller},
see \cite{PUZ:89}.

 A set formally larger than $\cl \Lambda$ is of interest, namely $\widehat\Lambda$ being the set of all accumulation points of $\{z_n(\alpha): \alpha\in\Sigma^d\}$ as $n\to\infty$. If our g.c.t. is in $\Omega$  being an RB-domain, see Section \ref{sec:bound}, or $f$ is just $R\circ g \circ R^{-1}$ defined only on $\Omega$, see Remarks below, then it is easy to see that $\cl \Lambda=\widehat\Lambda$. I do not know how general this equality is.

\end{definition}

{\bf Remarks.}
Given a Riemann map $R:\D\to \Om$ to a connected simply connected domain $\Om\subset\C$, (i.e. holomorphic bijection)
we can consider a branched covering map, say $g(z)=z^d$ on $\D$,
 and $f=R\circ g\circ R^{-1}$. Then, chosen $z\in \Om$  and $\g^j$ joining it with its preimages in $\Om$ (close to $\Fr\Om$) we can consider the respective tree $\sT$. %and its image tree $\sT'=R(\sT)$.
 Then instead of considering $R$ and its radial limit $\ov R$, we can consider
 the limit (along branches)  $z_\infty: \S^d\to\Fr\Om$.
 % for the tree $R(\sT)$.
%So $R(\sT)$ and its limit replaces Riemann map in its boundary behaviour and
 %It provides a symbolic structure of
  %a coding shift space enabling us to use
  This provides a structure of symbolic dynamics useful to verify stochastic laws.

   This is especially useful if considered measures come from $\partial \D$ via  $\ov R$, rather than being some equilibrium states for potentials living directly on $\Fr \Om$.
This is the case of harmonic measure $\omega$ which is the $\ov R_*$-image of a length measure $l$.
We can consider the lift of $l$ to $\Sigma^d$ via coding by the tree $\sT'=R^{-1}(\sT)$ and next its projection by $(z_\infty)_*$ to $\Fr\Om$.

Our g.c.t.'s are always available in presence of adequate holomorphic $f$, even in the absence of %a domain
$\Om$,
%(with "repelling" boundary),
i.e.~in the absence of a Riemann map. The tree with the coding it induces yields a discrete generalization/replacement of a Riemann map.

%The better tool: coding using a Markov partition seems not always available. ????ref???

%\smallskip

It was proved in \cite{PSkrzy} that $\sD$ is the whole $\Sigma^d$ except a ``thin'' set.
In particular, for a Gibbs measure $\nu$ for a H\"older potential, $z_\infty(\a)$ exists for $\nu$-a.e.~$\a$, hence the push forward measure $(z_\infty)_*(\nu)$ makes sense. Moreover, our codings $\z_\infty$ are always ``thin''-to-one. This is a discrete generalization of Beurling's Theorem concerning the boundary behaviour of Riemann maps. ``Thin'' means of zero logarithmic capacity type, depending on the properties of the tree
(the speed of the accumulation of $\gamma^j$ by critical trajectories; the speed does not matter if we replace ``thin'' by zero Hausdorff dimension).
In particular this coding preserves the entropies.

%Trees in $\Omega$ as above, or more generally in RB-domains, see Section \ref{sec:bound}, satisfy the thinnes assumptions. In general one can always find a g.c.t.

\smallskip

For appropriate $\nu\in \cM(\varsigma,\Sigma^d)$ and $\psi:\Sigma^d\to\mathbb{R}$ with $\int\psi\,d\nu=0$,
consider the {\it asymptotic variance} (of course one can consider spaces more general than $\Sigma^d$)
\begin{equation}\label{s2} %the {\it asymptotic variance}
\s^2=\sigma^2_\nu(\psi):=\lim_{n\to\infty}\frac1n\int (S_n\psi)^2\, d\nu.
\end{equation}

%=\int\psi^2\, d\nu +2\sum_{j=1}^\infty\int\psi\cdot(\psi\circ\varsigma^j)\, d\nu$.
%\end{equation}
%where the convergence and the equality hold under appropriate %mixing assumptions.

\begin{theorem}\label{tree-HD}
Let $\La$ be a quasi-repeller for a geometric coding tree
%$\sT(z,\g^1,...,\g^d)$
for a
holomorphic  map $f:U\to\ov\C$. Let $\nu$ be a $\varsigma$-invariant  Gibbs
measure on $\S^d$
for a H\"older continuous real-valued function $\phi$ on $\S^d$. Assume $P(\varsigma,\phi)=0$.
%(subtracting $P(\varsigma,u)$ we can assume it has pressure 0).
Consider $\mu:= (z_\infty)_*(\nu)$.

Then, for $\psi:= -\HD(\mu)(\log|f'|\circ z_\infty))-\phi$, we have $\int\psi\,d\nu=0$.

If the asymptotic variance $\s^2=\s^2_\nu(\psi)$ is positive,
then there exists a compact  $f$-invariant hyperbolic repeller $X$ being a subset of $\La$ such that
$\HD (X)>\HD(\mu)$. In consequence $\HD_{\hyp}(\La)>\HD(\mu)$ (defined after \eqref{Lyap_spec}).

If $\sigma^2=0$ then $\psi$ is cohomologous to 0. Then for each $x,y\in \cl\La$ not postcritical, if
$z=f^n(x)=f^m(y)$ for some positive integers $n,m$, the orders of criticality of $f^n$ at $x$ and $f^m$ at $y$ coincide. In particular all critical points in $\cl\Lambda$ are pre-periodic.
\end{theorem}

The latter condition %is non-empty in the case there are critical points in $\La$. It
happens only in special situations, see e.g. Theorem \ref{Z1-maximal} below.
See \cite{SzoUrbZdu:15}
for more details; $\phi$ lives there directly on $J(f)$, but it does not make substantial difference. See also Section \ref{sec:acc}.

\smallskip

Given a mapping $f:X\to X$, given two functions $u,v:X\to\mathbb{R}$ we call $u$ {\it cohomologous} to $v$ in class $\cC$ if there exists $h:X\to \mathbb{R}$ belonging to $\cC$ such that $u-v=h\circ f-h$.
An important \cite[Lemma 1]{PUZ:89} says that $\sigma^2=0$ above implies $\psi$ cohomologous to 0 in $L^2(\mu)$ and often in a smaller class depending on $\psi$ (Liv\v{s}ic type rigidity).

\smallskip

Notice that $\int\psi\,d\nu=-\HD(\mu)\chi(\mu)-\int\phi\,d\nu=-h_\mu(f)-\int\phi\,d\nu=
-h_\nu(\varsigma)-\int\phi\,d\nu=P(\varsigma,\phi)=0$.
Now, to prove Theorem \ref{tree-HD} note %that Lyapunov exponent $\chi(\mu):=\int \log |f'|\, d\mu$ is positive since
$2\chi(\mu)\ge h_{\mu}(f)=h_\nu(\varsigma)>0$, see \cite[Ruelle's inequality]{PUbook} (used also to 3)$\Rightarrow$5) in Section \ref{Hyperbolic potential}) and  \cite{Prz:85}. % and \cite{PUZ:89}.
So considering the natural extension of $(\Sigma^d,\nu,\varsigma)$ (here two-sided shift space) and Katok-Pesin theory, we find hyperbolic $X$ with
$\HD(X)\ge \HD(\mu)-\e$ for an arbitrary $\e>0$. Compare comments on shadowing in Section \ref{sec:other}.

\smallskip

$\bullet$ \ The positive $\s^2$ yields by Central Limit Theorem large fluctuations of the sums $\sum_{j=0}^{n-1}\psi\circ\varsigma^j$ from $n\int\psi\,d\nu$ (here 0), allowing to find $X$ with $\HD(X)>\HD(\mu)$.

A special care is needed to get $X\subset \La$,
see \cite{Prz:05} (originated in \cite{PZ:94}).

\smallskip

The above fluctuations were used by A. Zdunik to prove for constant $\phi$

\begin{theorem}[\cite{Z1}]\label{Z1-maximal} Let $f:\ov\C\to\ov\C$ be a rational mapping of degree
$d\ge 2$. If  $\sigma^2>0$, %$\log |f'|$ is not cohomologous to a constant,
then for $\mu_{\max}(f)$ the measure of maximal entropy (equal $\log d$),  $\HD(J(f))>\HD(\mu_{\max}(f))$. Otherwise, $f$ is postcritically finite with a parabolic orbifold,  \cite{Milnor}.
\end{theorem}

%Indeed, in the non-cohomological case, $\sigma^2\not=0$. Zdunik
She proved in fact the existence of a hyperbolic $X\subset J(f)$ satisfying
$\HD (X)>\HD(\mu_{\max}(f))$, hence $\HD_{\hyp}(J(f))>\HD(\mu_{\max}(f))$.
%Since in \cref{Z1-maximal} $\Lambda=J(f)$, it was obvious that  $X\subset\Lambda$.

\smallskip

$\bullet$ \  In the $\sigma^2=0$  case,
$v:J(f)\to \mathbb{R}$ satisfying the \underline{cohomology equation} $\log|f'|=v\circ f -v +\Const$ on $J(f)$ extends to a harmonic function beyond $J(f)$ (Liv\v{s}ic rigidity) giving this equality on the union of  real analytic curves containing $J(f)$ (called {\it real case}) or to $\ov\C$. In Theorem \ref{tree-HD} on $\Lambda$ and for the extension beyond, in Theorem \ref{Z1-maximal},
%restricted to $J(f)$.
the ``orders'' of growth of $-\log |(f^n)'|$ at $x$ and of $-\log |(f^m)'|$ at $y$ must by cohomology equation  be equal to the ``order'' of growth of $v$ at $z$, so they
must coincide (a phenomenon ``conjugated'' to the presence of an invariant line field).
This implies parabolic orbifold for Theorem \ref{Z1-maximal}.

\smallskip

Theorem \ref{Z1-maximal} applied to a polynomial $f$ with connected Julia set, by

\noindent $\HD(\mu_{\max}(f))=1$ \cite{Manning},
implies the following celebrated result:

\begin{theorem}[ A. Zdunik \cite{Z1}]\label{Z1-polynomial}
 For every polynomial $f$ of degree at least 2, with connected Julia set,
either $J(f)$ is a circle or an interval or else
it is fractal, namely $\HD(J(f))>1$.
\end{theorem}

\section{Boundaries, radial growth, harmonic vs Hausdorff}\label{sec:bound}

For polynomials with connected Julia set the measure $\mu_{\max}(f)$ coincides with harmonic measure $\omega$ (viewed from $\infty$). % and the equality $\HD(\omega)=1$ holds.
%(this equality is in fact  a general phenomenon  true for all  simply connected domains even in absence of $f$\cite{Makarov}).
This led to another   proof of Theorem \ref{Z1-polynomial}, especially the $\sigma^2=0$ part, see \cite{Z2}, in the language of boundary behaviour of Riemann map and  harmonic measure (compare also  model Theorem \ref{z2+c} ).

Theorem \ref{Z1-polynomial} has been strengthened  from this point of view in
\cite{Prz:06}, preceded by \cite{PZ:94}, as follows.

\begin{theorem}\label{HD>1}
Let $f:\ov\C\to\ov\C$ be a rational map of degree at least 2 and $\Om$ be
a  simply connected immediate basin of attraction to an attracting periodic orbit  (that is a connected component of the set attracted to the orbit, intersecting it).
Then, provided $f$ is not a finite Blaschke product in some holomorphic coordinates, or
a two-to-one holomorphic factor of a Blaschke product,
$\HD_{\hyp}(\Fr\Om)>1$.
\end{theorem}

The novelty was to show how to ``capture'' a large hyperbolic $X$ in $\Fr\Om$ in the case it was not the whole $J(f)$.

In fact the following ``local'' version of this theorem was proved in \cite{Prz:06}

\begin{theorem}\label{local-HD>1}
Assume that $f$ is defined and holomorphic on
a neighbourhood $W$ of $\Fr\Om$, where $\Omega$ is a connected, simply connected domain in $\ov\C$ whose
boundary has at least 2 points.   We assume that $f(W\cap \Om)\subset \Om$,\
$f(\Fr\Om)\subset\Fr\Om$ and $\Fr\Om$ repels
to the side of $\Om$, that is $\bigcap_{n=0}^\infty
f^{-n}(W\cap\cl\Om)=\Fr\Om$. Then either $\HD_{\hyp}(\Fr(\Om))>1$ or
$\Fr\Om$ is  a real-analytic Jordan curve or arc.
\end{theorem}

$\Om$ with $f$ as above has been called an {\it RB-domain} (repelling boundary), introduced in \cite{Prz:86, PUZ:89}. Theorem \ref{local-HD>1} (at least the $\sigma^2>0$ part) follows directly from Theorem \ref{tree-HD}. Let $R:\D\to\Omega$ be a Riemann map  and $g: W'\to\D$ be defined by $g:=R^{-1}\circ f \circ R$ on $W'=R^{-1}(W\cap\Om)$.
We consider a g.c.t. ${\sT}=\sT(z,\g^1,\dots,\g^d)$ with $z$ and $\g^j$ in $W\cap\Om$,  sufficiently close to $\Fr\Omega$ that the definition makes sense, and with $d=\deg f|_{W\cap\Om}$,
(the situation is the same as in Remarks in  Section \ref{sec:gct} above, but the order of defining $f$ and $g$ is different). Consider the g.c.t. ${\sT}'=R^{-1}(\sT)$.  The function $g$ extends holomorphically beyond the circle $\partial \D$ and it is expanding. % (compare Proof of \cref{z2+c}).
Hence
$\phi:\Sigma^d\to \mathbb{R}$ defined by $\phi(\alpha)=-\log |g'|\circ (R^{-1}(z))_\infty(\a)$ for the tree ${\sT}'$ is H\"older continuous. Let $\nu=\nu_\phi$.

Note that here $P(\phi)=0$, e.g. since by expanding property of $g$ on $\partial\D$ there exists
$\hat l\in\cM(g)$, equivalent to length measure $l$ (a.c.i.m.).
Then $\nu$ is the lift of $\hat l$ to $\Sigma^d$ with the use of $\sT'$.
%$\int\psi\,d\nu=-\HD(\mu)\int \log |f'|\circ z_\infty \,d\mu - \phi = -\h_\mu(f)+\int\log|g'|\,d\hat l =  -\h_{{\hat l}}(g)+\int\log|g'|\,d{\hat l}=0$,
%where the measure ${\hat l}=(R^{-1}(z)_\infty)_*(\nu)$ is $g$-invariant, equivalent to the length measure on $\partial \D$. By definition $\mu=\ov R_*({\hat l})$, where $\ov R$ stands for the radial extension of $R$ a.e.,
 Note that our  $\mu=z_\infty(\nu)$ is equal to $\hat\omega=\ov{R}_*(\hat l)$ which is $f$-invariant, equivalent to harmonic measures $\omega$ on $\Fr \Om$ viewed from $\Om$. %We used the equality of entropies $\h_{{\hat l}}(g)=\h_{\mu}(f)$, see \cite{Prz:85} and \cite{Prz:86} %or in a general setting \cite{Mane}.

 Note that $\HD(\hat\omega)=1$ due to Ma\~n\'e's equality, \eqref{mane},  $h_{\hat\omega}(f)=h_{\hat l}(g)$, see \cite{Prz:85, Prz:86}, and  the equality of Lyapunov exponents $\int\log|f'|\,d\hat\omega = \int \log |g'|\,d{\hat l} > 0$. The latter equality holds due to the equality for almost every $\zeta\in \partial\D$:

\begin{equation}\label{radial1}
\lim_{r\to 1}\frac {\log |(f^n)'(R(r\z))|-\log |(g^n)'(r\z))}{\log (1-r)} =\lim_{r\to 1}
\frac{-\log |R'(r\z)|}
{\log (1-r)}=0.
\end{equation}
The first equality is proved using $f\circ R=R\circ g$ in $\D$, first applying $R$ close to $\partial \D$, next by iterating $f$ applying $R^{-1}$ well inside $\Om$, finally  iterating $g$ back. The latter equality relies on the harmonicity of $\log|R'|$ allowing to replace its integral on circles by its value at the origin. For details see \cite{Prz:86}. Remind however that in fact $\HD(\omega)=1$ holds in general, see \cite{Makarov}.

%$P(\s,u)=h_\nu(\s)-\int \log |g'|\,d( R^{-1}(z)_\infty)_*(\nu) )=h_\mu(\s)-\int\log|f'|\,d\mu=0$.
%Here $R^{-1}(z)_\infty)_*(\nu)$ and $\mu$ are $g$ and $f$ invariant measures equivalent to harmonic ones on $\partial \D$ and $\Fr\Om$ respectively.

The sketch of Proof of Theorem \ref{local-HD>1} for $\sigma^2>0$ is over. That $\sigma^2=0$ implies the analyticity of $\Fr\Om$ was already commented at the beginning of this Section.
%and referred to \additional considerations are needed.

\section{Law of Iterated Logarithm refined versions}\label{LIL}

Applying Law of Iterated Logarithm (abbr. LIL) to $\psi: \Sigma^d\to \mathbb{R}$   the fluctuations of $S_n\psi$ from 0 which follow lead to, see \cite{PUZ:89} and \cite{PUbook}, %(where we restricted with $\La$ to hyperbolic repellers),

\begin{theorem}\label{LIL-refined-HD}
In the setting of Theorem \ref{tree-HD}
if $\s^2=\s^2_\nu(\psi)>0$, for $c(\mu):=\sqrt{2\s^2/\chi(\mu)}$, $\kappa:=\HD(\mu)$ and
$\a_c(r):=r^\kappa\exp(c\sqrt{\log 1/r\log\log\log 1/r})$

1) $\mu \bot H_{\a_c}$, that is singular with respect to the refined Hausdorff measure, \cite[Section 8.2]{PUbook} for the gauge function $\a_c$), for all $0<c<c(\mu)$;

2) $\mu \ll H_{\a_c}$, that is absolutely continuous, for all $c>c(\mu)$.

\end{theorem}

%a) is immediate. Uniqueness of measure maximizing dimension follows from Ledrappier Dobbs ??? nie tu ???

%The proof relies on the observation that, assumed $\int\psi d\nu=0$, LIL yields

Indeed, substituting in LIL $n\sim (\log 1/r_n)/\chi(\mu)$
for $r_n = |(f^n)'(z)|^{-n}$, we get for $\mu$-a.e. $z$ %and a sequence of $n$'s, %see \cite[Lemma {PUZ:89}
\begin{equation}\label{sing-vs-cont}
\limsup_{n\to\infty}\frac{\mu(B(z,r_n)}{\a_c(r_n)}=\infty%{  r_n^\kappa\exp(c\sqrt{\log 1/r_n\log\log\log 1/r_n})}=\infty
\ {\rm for} \ 0<c<c(\mu) \
\ \ {\rm and}\ \ \dots =0 \ {\rm for} \ c>c(\mu).
 \end{equation}
% in the latter case.  %in the expression \ref{sing-vs-cont}
 %after  a truncation a small part of $\mu$.
%replaced by $\mu_1=\mu|_A$ where $\mu(X\setminus A)$ is arbitrarily small.

This is called the Refined Volume Lemma,  \cite[Section 4]{PUZ:89} and, the harder case: $c>c(\mu)$, \cite[Section 5]{PUZ:91}. %(Difficulties were caused by passing from $\nu$ to $\mu$.)

\medskip

%By building a g.c.t. in a simply connected RB-domain $\Om$ with a holomorphic map $f$ (in particular in any simply-connected immediate  basin of attraction of a rational map)

We can apply the assertion of Theorem \ref{LIL-refined-HD} for $\mu=\hat\omega\in\cM(f,\Fr\Om)$ equivalent to a harmonic measure $\omega$ as in Section \ref{sec:bound}.

\smallskip
%Additionally if $\psi$ is cohomologous to a constant (here 0) (i.e.~$\sigma^2=0$), $\Fr\Omega$ is an analytic Jordan curve or an analytically embedded interval, see \cite{Z2}.

This yields refined information about the radial growth of the derivative of Riemann maps, following the proof of \eqref{radial1}:

\begin{theorem}\label{radial growth}
 Let  $\Om$ be a simply connected RB-domain in $\ov\C$ with non-analytic boundary and $R:\D\to\Om$ be a Riemann map. Then there exists $c(\Om)>0$ such that for Lebesgue a.e. $\z\in\partial \D$
 \begin{equation}\label{radial}
 G^+(\z):=\limsup_{r\to 1}\frac{\log |R'(r\z)|}{\sqrt{\log(1/1-r)\log\log\log(1/1-r)}}=c(\Omega).
\end{equation}
Similarly  $G^-(\z):=\liminf \dots= -c(\Omega)$. Finally $c(\Omega)=c(\hat\omega)$ in
Theorem \ref{LIL-refined-HD}.
\end{theorem}

%This corollary is immediate if $\Om$ is a Jordan domain and $f|_{\Fr\Om}$ is expanding, see e.g. \cite{PUbook}[Theorem 9.5.7]

In fact Theorem \ref{LIL-refined-HD} for $\mu=\hat\omega$ and Theorem \ref{radial growth} hold for every connected, simply connected open $\Om\subset\C$, together with $c(\Omega)=c(\hat\omega)$. No dynamics is needed. Of course one should add to both definitions
${\rm ess}\sup$ over $\z\in\partial\D$ and over $z\in \Fr\Om$ (for $c(z)=c(\omega)$ calculated from \eqref{sing-vs-cont}, see \cite[Th. 8.6.1]{PUbook}  ) respectively, since in the absence of ergodicity these functions need not be constant. See \cite[Th. VIII.2.1 (a)]{GaMa} and references to Makarov's breakthrough papers therein, in particular \cite{Makarov}.

% Lemma 3.8 in Makarov, N. G. Probability methods in the theory of conformal mappings. (Russian) Algebra i Analiz 1 (1989), no. 1, 3--59; translation in Leningrad Math. J. 1 (1990), no. 1, 1�56. Also Pommerenke

There is a universal Makarov's upper bound $C_{\rm M}<\infty$ for all $c(\Omega), c(\hat\omega)$'s in \eqref{radial}.
The best upper estimate I found in literature is $C_{\rm M}\le 1.2326$,  \cite{HK}. I proved in \cite{Prz:89} a much weaker estimate, using a natural method of representing $\log |R'|$ by a series of weakly dependent random variables leading to a martingale on $\partial\D$, thus satisfying LIL. Unfortunately consecutive approximations resulted with looses in the final estimate.

%This yields the estimate $c(\hat\omega)\le 3C_{\rm M}$ in \cref{LIL-refined-HD} when dynamics not assumed to exist and there is a belief the factor 3 can be skipped, as in presence of $f$ above. See \cite{Ivrii} for the case $R$ extends quasi-conformally to $\C$.

%In absence of $f$ one can express $\sigma^2

%=====================================================================
For a holomorphic expanding repeller $f:X\to X$ and a  H\"older continuous potential $\phi:X\to X$, the asymptotic variance for the equilibrium state $\mu=\mu_{t_0\phi}$ for every $t_0\in \mathbb{R}$ satisfies Ruelle's formula (see \cite{PUbook}):
\begin{equation}\label{second_derivative}
\sigma^2_\mu(\phi-\int\phi\,d\mu)=\frac{d^2P(t\phi)}{dt^2} \biggl |_{t=t_0}.
\end{equation}
{\bf Question.} Does \eqref{second_derivative} hold for all rational maps and hyperbolic potentials on Julia sets? For all simply connected RB-domains, $f:\Fr\Om\to\Fr\Om$ and $\mu=\hat\omega$?

\smallskip

%For $X=\Fr\Om$ for $\Om$ being
For a simply connected RB-domain $\Om$ for $f$ and for $\phi=-\log|f'|$, if $g(z)=z^d$ (e.g.~$\Om$ being  the basin of $\infty$ for a polynomial $f$), one considers the {\it integral means spectrum} depending only on $\Om$,
\begin{equation}
\beta_{\Om}(t):= \limsup_{r\to 1}\frac1{|\log (1-r)|}\log\int_{\z\in\partial\D}|R'(r\z)|^t\, |d\z|
\end{equation}
which happens to satisfy $\beta_{\Om}(t)=t-1+\frac{P(t\phi)}{\log d}$, see e.g.~\cite[Eq. (9.6.2.)]{PUbook}.

For $t_0=0$ we have $\mu=\hat\omega$ and the left hand side of \eqref{second_derivative} can be written as $(\frac12 \log d)\sigma^2(\log R')$,   %an "asymptotic variance"
see \eqref{s2} and \eqref{radial1}, where
$$\sigma^2(\log R'):=
 \limsup_{r\to 1} \frac{\int_{\partial\D}|\log R'(t\z)|^2\, |d\z|}{-2\pi \log(1-r)|}.
$$  %,
 So \eqref{second_derivative} changes to $\sigma^2(\log R')=2\frac{d^2\beta_{\Om}(t))}{dt^2}|_{t=0}$,
 compare \cite{Ivrii}.
 It has an analytic, non-dynamical, meaning. It is also
 related to the Weil-Petersson metric, see \cite{McM}.

%==========================================================================

%Then the formula \eqref{second_derivative} takes the form
%$\frac{\sigma^2_\omega(\phi)}{\log d}=\frac{d^2\b(t\phi)}{dt^2}|_{t=0}.$
%where on the left hand side we have

%\noindent $\sigma^2_\beta:= and \eqref{s2}.

%$$
%\beta(t):=\limsup_{r\nearrow 1}\frac{\log\int_{\partial\D}|R'(rz)| dl(z)}{-\log(1-r)}
%$$
%If it is an RB-domain for $f$, with $g(z)=R^{-1}fR(z)=z^d$, we have
%\begin{equation}
%\beta(t)=t-1+\frac{P(f,-t\log|f'|)}{\log d}.
%\end{equation}
%For $\sigma^2:=
%it is compatible with \eqref{second_derivative}

\section{Accessibility}\label{sec:acc}

Let us recall the following  theorem from \cite{Prz:94}.

\begin{theorem}\label{access}

Let %as in \cref{tree-HD}
$\La$ be a quasi-repeller for a geometric coding tree
%$\sT(z,\g^1,...,\g^d)$
for a
holomorphic  map $f:U\to\ov\C$. Suppose that
\begin{equation}\label{shrink}
\diam (\g_n(\a)) \to 0 , \ \hbox{as} \  n\to
\infty
\end{equation}
uniformly with respect to $\a\in\S^d$.
Then every {\it good} $q\in\widehat\La$ (defined in Section \ref{sec:gct}) is a limit of a convergent
branch $b(\a)$. So $q\in\Lambda$. In particular, this holds for every $q$ with $\underline{\chi}(q)>0$ and the local backward inviariance (explained below).
\end{theorem}
For the definition of ``good'', see \cite[Definition 2.5]{Prz:94}. It roughly says that there are many integers $n$ (positive lower density) for which $f^n$ properly map small domains $D_{n,0}$ in $U$ close to $q$ onto large $D_n\subset U$, giving  ``telescopes'' ${\rm Tel}_k$ with ``traces'' $D_{n_k,0}\subset D_{n_{k-1},0}\subset\dots\subset D_{n_1,0}\subset D_0$; for each $k$ the choices may be different.
A part of this condition that $D_{n,0}\subset U$ can be called   a ``local backward invariance'' of $U$ along the forward trajectory of $q$.

When $U$ is an immediate basin of attraction of an attracting fixed point for a rational map $f$ or just an RB-domain then this theorem asserts that $q$ is an endpoint of a continuous curve in $U$.  This is a generalization of the Douady-Eremenko-Levin-Petersen theorem where $q$ is a repelling periodic point and the domain is completely invariant, e.g. basin of attraction to $\infty$ for $f$ a polynomial.

Due to this theorem we can prove that invariant measures of positive Lyapunov exponents lift to $\Sigma^d$.
More precisely, the following holds:

%{\bf Dalej osiagalnosc pk $\chi>0$.}

\begin{corollary} Every non-atomic hyperbolic probability measure $\mu$ \ (i.e. $\chi(\mu)>0$), on $\widehat\La$,  is  the $(z_\infty)_*$ image of a probability $\varsigma$-invariant measure $\nu$ on $\Sigma^d$, assumed \eqref{shrink},  $\sT$ has no self-intersections and else $\mu$-a.e. local backward invariance of $U$,.
In particular, $\nu$ exists for every RB-domain which is completely (i.e. backward) invariant. %where these assumptions always hold
\end{corollary}

 \begin{proof} (the lifting part missing in \cite{Prz:94} and \cite{Prz:06}). By Theorem \ref{access} $\mu$ is supported on $\Lambda$ i.e. on $z_\infty(\sD(z_\infty))$. The lift of $\mu$ to $\mu'$ on the pre-image $\cB'$ under $z_\infty$ of the Borel $\sigma$-algebra of subsets of $\La$ can be extended to a $\varsigma$-invariant $\nu$ on $\cB$ the Borel $\sigma$-algebra of the subsets of $\Sigma^d$ by using the fact
that the set of at least triple points (limit points of at least three infinite branches of $\sT$) is countable,
hence $z_\infty^{-1}(x)$ of $\mu$-a.e $x$ contains at most 2 points. More precisely, let $A_1$ be the set
of points having one $z_\infty$-preimage, $A_2$ two preimages.
They are both $f$-invariant (except measure 0), so are their $z_\infty$-preimages $A'_1$ and $A'_2$ under $\varsigma$.
We extend  $\mu'$ by distributing conditional measures on the two points preimages of points in $A_2$ half-half and Dirac on one point preimages.

\end{proof}

This allows to conclude Theorem \ref{LIL-refined-HD} (a part relying on CLT) and Theorem \ref{tree-HD}
for equilibrium states for rational maps and H\"older
potentials on $J(f)$ by lifting $\mu_\phi$ to $\Sigma^d$ as in \cite{Prz:06}.
%see comments to \cref{equi-hyp}.
However, this
%lifting of $\mu_\phi$ to $\Sigma^d$
seems useless since the proof of CLT in \cite{Prz:06} is done directly on $J(f)$ (seemingly also for LIL, for which one should however refer to the proofs in \cite{PUZ:89}) and there are direct proofs of LIL in  \cite{RivLi2} and \cite{SzoUrbZdu:15}.

\subsection*{Acknowledgments}
Thanks to H. Hedenmalm, O. Ivrii, J. Rivera-Letelier, M. Sabok, M. Urba\'nski, and A. Zdunik  for   comments and corrections.

\bibliographystyle{amsplain}

\end{document}